\documentclass{amsart}
\usepackage{graphicx}
\usepackage{appendix}
\usepackage{amsfonts}
\usepackage{listings}
\usepackage{color}
\usepackage{float}
\usepackage{algorithm}
\usepackage{algorithmic}
\definecolor{codegreen}{rgb}{0,0.6,0}
\definecolor{codegray}{rgb}{0.5,0.5,0.5}
\definecolor{codepurple}{rgb}{0.58,0,0.82}
\definecolor{backcolour}{rgb}{0.95,0.95,0.92}

\lstdefinestyle{mystyle}{
    backgroundcolor=\color{backcolour},   
    commentstyle=\color{codegreen},
    keywordstyle=\color{magenta},
    numberstyle=\tiny\color{codegray},
    stringstyle=\color{codepurple},
    basicstyle=\ttfamily\footnotesize,
    breakatwhitespace=false,         
    breaklines=true,                 
    captionpos=b,                    
    keepspaces=true,                 
    numbers=left,                    
    numbersep=5pt,                  
    showspaces=false,                
    showstringspaces=false,
    showtabs=false,                  
    tabsize=2,
    frame=single                     
}

\newtheorem{tm}{Theorem}

\newtheorem{defi}{Definition}

\newtheorem{rem}{Remark}
\newtheorem{rems}{Remarks}
\newtheorem{lm}{Lemma}
\newtheorem{ex}{Example}
\newtheorem{cor}{Corollary}
\newtheorem{prop}{Proposition}
\newtheorem{nota}{Notation}

\begin{document}
\title{Midpoints and critical points}
\author{Yousra Gati and Vladimir Petrov Kostov}

\address{Universit\'e de Carthage, EPT-LIM, Tunisie}
\email{yousra.gati@gmail.com}
\address{Universit\'e C\^ote d’Azur, CNRS, LJAD, France}
\email{vladimir.kostov@unice.fr}
\begin{abstract}
  For a degree $5$ real polynomial with roots $x_1\leq \cdots \leq x_5$ and
  roots $\xi_1\leq \cdots \leq \xi_4$ of its derivative, we set
  $z_j:=(x_j+x_{j+1})/2$, $1\leq j\leq 4$. We prove that one
  cannot have at the same time
  $\min_{1\leq j\leq 3}(z_{j+1}-z_j)\geq \min_{1\leq j\leq 3}(\xi_{j+1}-\xi_j)$ and
  $\max_{1\leq j\leq 3}(z_{j+1}-z_j)\geq \max_{1\leq j\leq 3}(\xi_{j+1}-\xi_j)$.
  The result settles a general question about midpoints and critical points
  of hyperbolic polynomials.\\ 


\end{abstract}
\maketitle

\section{Introduction}

We consider degree $d$ {\em hyperbolic} polynomials, i.~e. real uni-variate
polynomials with $d$ real roots. Such a monic polynomial can be
represented in the form 

$$P(x)=(x-x_1)(x-x_2)\cdots (x-x_d)~,~~~\, x_1\leq x_2\leq \cdots
\leq x_d~.$$
In most cases we are interested in {\em strictly hyperbolic} polynomials,
i.~e. with all
roots distinct. For a strictly hyperbolic polynomial, the classical Rolle's
theorem states about the critical points $\xi_j$ (i.~e. roots of $P'$)
that $\xi_j\in (x_j,x_{j+1})$, $j=1$, $\ldots$, $d-1$.

It is natural to compare the roots $\xi_j$ with the
{\em midpoints} $z_j:=(x_j+x_{j+1})/2$. For $d=2$, it is clear that $\xi_1=z_1$.
For $d\geq 3$, the following lemma holds true:

\begin{lm}\label{lmbasic}
  Suppose that $d\geq 3$. Then $\xi_1<z_1$ and $\xi_{d-1}>z_{d-1}$, so
  $\xi_{d-1}-\xi_1>z_{d-1}-z_1$.
\end{lm}

\begin{proof}
  Indeed, $(P'/P)(z_{d-1})=\sum_{j=1}^{d-2}1/(z_{d-1}-x_j)>0$ and
  $\lim_{x\rightarrow x_d^-}(P'/P)(x)=-\infty$, so $\xi_{d-1}\in (z_{d-1},x_d)$.
  In a similar way one shows that $\xi_1\in (x_1,z_1)$.
  \end{proof}

In what follows we deal with the quantities

$$\begin{array}{cclccl}
  m&:=&\min_j(z_{j+1}-z_j)~,&M&:=&
  \max_j(z_{j+1}-z_j)~,\\ \\
  \tilde{m}&:=&\min_j(\xi_{j+1}-\xi_j)~,&
  \tilde{M}&:=&
  \max_j(\xi_{j+1}-\xi_j)~,\\ \\ j&=&1,~\ldots ,~d-2~.\end{array}$$
As $z_{j+1}-z_j=(x_{j+2}-x_{j+1})/2+(x_{j+1}-x_j)/2=(x_{j+2}-x_j)/2$, it is clear that

\begin{equation}\label{equMRiesz}
  m\geq m^{\dagger}:=\min_{j=1,\ldots ,d-1}(x_{j+1}-x_j)~~~\, {\rm and}~~~\, 
M\leq M^{\dagger}:=\max_{j=1,\ldots ,d-1}(x_{j+1}-x_j)~.\end{equation}
In this text we consider the more interesting question to
compare $m$ with $\tilde{m}$
and $M$ with $\tilde{M}$. For a hyperbolic polynomial $P$,
we check whether the following
two inequalities hold true:

\begin{equation}\label{eqineqLR}
  (L)~:~m\leq \tilde{m}~~~\, \, \, {\rm and}~~~\, \, \,
  (R)~:~\tilde{M}\leq M~.
  \end{equation}
\begin{defi}
  {\rm We say that the polynomial $P$ {\em realizes} the case $L+$
    (resp. $R-$) if
the inequality $(L)$ holds true (resp. if $(R)$ fails); similarly for the
cases $L-$ and $R+$. We say that $P$ realizes the case $L-R+$ if the inequality
$(L)$ fails while the inequality $(R)$ holds true (and similarly for the cases
$L-R-$, $L+R-$ and $L+R+$).

Accordingly, we say that a case $L\pm$, $R\pm$ or $L\pm R\pm$
is {\em realizable}
  if there exists a strictly hyperbolic polynomial realizing this case.}
  \end{defi}

With the present paper we settle definitely the question about the
realizability of all cases $L\pm R\pm$, for any degree $d\geq 3$.
The results are presented in the following table:

$$\begin{array}{lccccc} 
  d&&L+R+&L+R-&L-R+&L-R-\\ \\ 3&&{\rm No}&{\rm Yes}&{\rm No}&{\rm No}\\ \\
  4&&{\rm Yes}&{\rm Yes}&{\rm No}&{\rm No}\\ \\
  5&&{\rm Yes}&{\rm Yes}&{\rm\bf No}&{\rm Yes}\\ \\ \geq 6&&{\rm Yes}&
  {\rm Yes}&Yes&{\rm Yes}\end{array}$$
In the present text we justify the answer {\bf ``No''} for the case $L-R+$
with $d=5$. All other cases except the {\em Yes}-case $L-R+$ with $d=6$ are
settled in~\cite{DiKo}. The positive answer to $L-R+$ with $d=6$
is given in \cite{GaKoTa} (see also \cite[Example~1.8.1]{KoDeGruyter}),
where a numerical method for rapid search
of a realizing polynomial is developed. The method proposes the following  
polynomial realizing the case $L-R+$ with $d=6$:

$$x^6-1.6x^5+0.53x^4+0.122578x^3-0.03793509x^2-0.0025040322x+0.000600530112~.$$

The above table shows that the case $L-R+$ with $d=5$ is on the border
between the Yes- and No-answers.
It requires to consider a three-parameter family of polynomials (see
Subsection~\ref{subsecparam}). Moreover, there are polynomials in this family
which
realize the cases $L-R-$, $L+R-$ and $L+R+$ (see
Example~\ref{ex2possibilities}), therefore when justifying the
answer {\bf ``No''} one cannot adopt one and the same approach
to all
polynomials of the family. In this sense the case $L-R+$ with $d=5$ is the
most difficult and the most interesting. So the main result
of this paper reads:  

\begin{tm}\label{tm1}
  For $d=5$, the case $L-R+$ is not realizable.
\end{tm}

To understand how close to realizability the case $L-R+$ is becomes clear from
part~(2) of Example~\ref{ex2possibilities}. In the next section we explain
how Theorem~\ref{tm1} is inscribed within a larger
span of results concerning critical points and midpoints. In
Section~\ref{secmethod} we describe the methods used in the proof of
Theorem~\ref{tm1} and we give a plan of the rest of the paper.

\section{Comments}

The question about realizability of the cases $L\pm R\pm$ with regard
to the quantities $m$, $M$, $\tilde{m}$ and $\tilde{M}$ is inspired by 
a classical result of Marcel Riesz and by a question
concerning real entire functions of order $1$ having only real
zeros, asked recently by David Farmer and Robert Rhoades.
We remind that the set of real
entire functions of order at most $2$ with real zeros only
is called the {\em Laguerre-P\'olya class}. We call such functions
$\mathcal{LP}$-{\em functions}. We limit ourselves to discussing
only $\mathcal{LP}$-functions of order $1$ since only they are
concerned by the Farmer-Rhoades problem.
We suppose that the multiplicity of the zeros is at most $2$, because
in the presence of a triple root one obtains $m=\tilde{m}=0$. 
For a hyperbolic polynomial, Riesz has shown that $m^{\dagger}>m$
whenever the roots are distinct
(i.~e. the polynomial is strictly hyperbolic), see
\cite{FaRh}, \cite{Ob} and \cite{St}.

To cite the result of Farmer-Rhoades we need to remind that
$\mathcal{LP}$-functions are uniform limits on compact sets
of sequences of hyperbolic polynomials. Within this class there is the
subclass of $\mathcal{LP}I$-{\em functions} (of order~$1$) which are such limits of
hyperbolic polynomials with all roots of the same sign. For $f$ an
$\mathcal{LP}I$-function with
zeros $x_k$, arranged in increasing order, and for $a\in \mathbb{R}$,
denote by $\xi <\eta$
two consecutive zeros of the function $f'+af$. It is proved
in \cite{FaRh} (see Theorem~2.3.1 therein) that

$$\inf \{ x_{k+1}-x_k\} \leq \eta -\xi \leq \sup \{ x_{k+1}-x_k\}~.$$
In particular, for $a=0$, both inequalities hold true.

In \cite{DiKo} examples of entire functions of order $1$ (but not
$\mathcal{LP}I$-functions) are given for which both inequalities hold true
or one of them fails. 

Given a strictly hyperbolic polynomial, it is a priori clear that its
critical points cannot be too close to its roots. In the aim to make
Rolle's theorem more precise P.~Andrews (see~\cite{An}) has proved that
for $d\geq 2$, one has

\begin{equation}\label{equAn}
  \frac{1}{d-j+1}<\frac{\xi_j-x_j}{x_{j+1}-x_j}<\frac{j}{j+1}~,~~~\,
  j=1,~\ldots ,~d-1~;
\end{equation}
these inequalities are only necessary, but not sufficient conditions.
We mention also Alan Horwitz' results \cite{Ho1}-\cite{Ho2} in this direction.

Boris Shapiro and Michael Shapiro have considered {\em pseudopolynomials}
of degree $d$, i.~e.
smooth functions whose $d$th derivatives vanish nowhere. For the zeros
$x_k^{(j)}$ of the $j$th derivative $x_1^{(j)}\leq \cdots \leq x_{d-j}^{(j)}$,
one has $x_k^{(j-1)}\leq x_k^{(j)}\leq x_{k+1}^{(j-1)}$. In \cite{ShSh}
necessary and sufficient conditions are given how one should choose the
real numbers $x_k^{(j)}$, $j=0$, $1$, $2$, $k=1$, $\ldots$, $3-j$, so that they
could be the zeros of a degree $3$ pseudopolynomial. (For $d=2$, every triple
$x_1^{(0)}<x_1^{(1)}<x_2^{(0)}$ is the triple of zeros of a degree $2$
pseudopolynomial.)

For a strictly hyperbolic polynomial, one can consider the arrangement on the
real axis of all its roots and the roots of all its non-constant derivatives
together. One considers mainly {\em generic} arrangements, i.~e. with no
equality between any two of these $n(n+1)/2$ roots. For $n\leq 3$, any such
arrangement (compatible with Rolle's theorem) is realizable by a strictly
hyperbolic polynomial. For $n=4$, two out of twelve arrangements are not
realizable by strictly hyperbolic polynomials (see~\cite{Ander}),
yet they are realizable by
pseudopolynomials (see~\cite{Ko}). However for $n=5$, there are arrangements
which are not realizable by the roots of strictly hyperbolic polynomials or
pseudopolynomials (and their derivatives), see~\cite{Ko1}. For $n=5$,
the exhaustive answer to the question which generic arrangements 
are realizable by strictly hyperbolic polynomials or by
pseudopolynomials can be found in \cite{Ko2},
\cite{Ko3} and \cite{Ko4}; in these articles pseudopolynomials are called
{\em polynomial-like functions}.

\section{The method of proof\protect\label{secmethod}}

\subsection{The parameter space\protect\label{subsecparam}}

Given a degree $5$ hyperbolic polynomial,
one can perform a linear transformation of the
independent variable after which the smallest root equals
$0$ and the largest equals $1$. The new polynomial realizes the same case
$L\pm R\pm$ as the initial one. That's why we consider the three-parameter
family of polynomials

\begin{equation}\label{equF}
  F:=x(x-a)(x-b)(x-c)(x-1)~,
\end{equation}
whose roots satisfy the conditions $0\leq a\leq b \leq c\leq 1$.
(In some of the lemmas we use other parametrizations as well.) 
The simplex

$$\mathbb{R}^3~\cong ~Oabc~\supset ~\tilde{S}~:~ 0 \leq a \leq b \leq c \leq 1$$
is invariant under the change of variable $x \mapsto 1-x$.
The roots change as follows:

$$a\mapsto 1-c~,~~~\, b\mapsto 1-b~~~\, {\rm and}~~~\, c\mapsto 1-a~.$$
The (hyper)planes $a+c=1$ and $b=1/2$ are invariant under this change.
The latter divides the simplex $\tilde{S}$ into two congruent parts.
It suffices to prove
that the case $L-R+$ is not realizable in one of them. We choose
this to be the part $S$ containing the vertices
$U:=(0,0,0)$ and $V:=(0,0,1)$, see Fig.~\ref{Yousrafig3D}. (We explain below
the role of the plane $c=1/3.1$~.) 
Its other vertices are contained in the plane
$b=0.5$. They are

\begin{equation}\label{equvertices}
I=(0.5,~0.5,~0.5)~,~~~\,  J=(0.5,~0.5,~1)~,~~~\, K=(0,~0.5,~1)~~~\,
{\rm and}~~~\, 
G=(0,~0.5,~0.5)~.\end{equation}

\begin{figure}[htbp]
\centerline{\hbox{\includegraphics[scale=0.5]{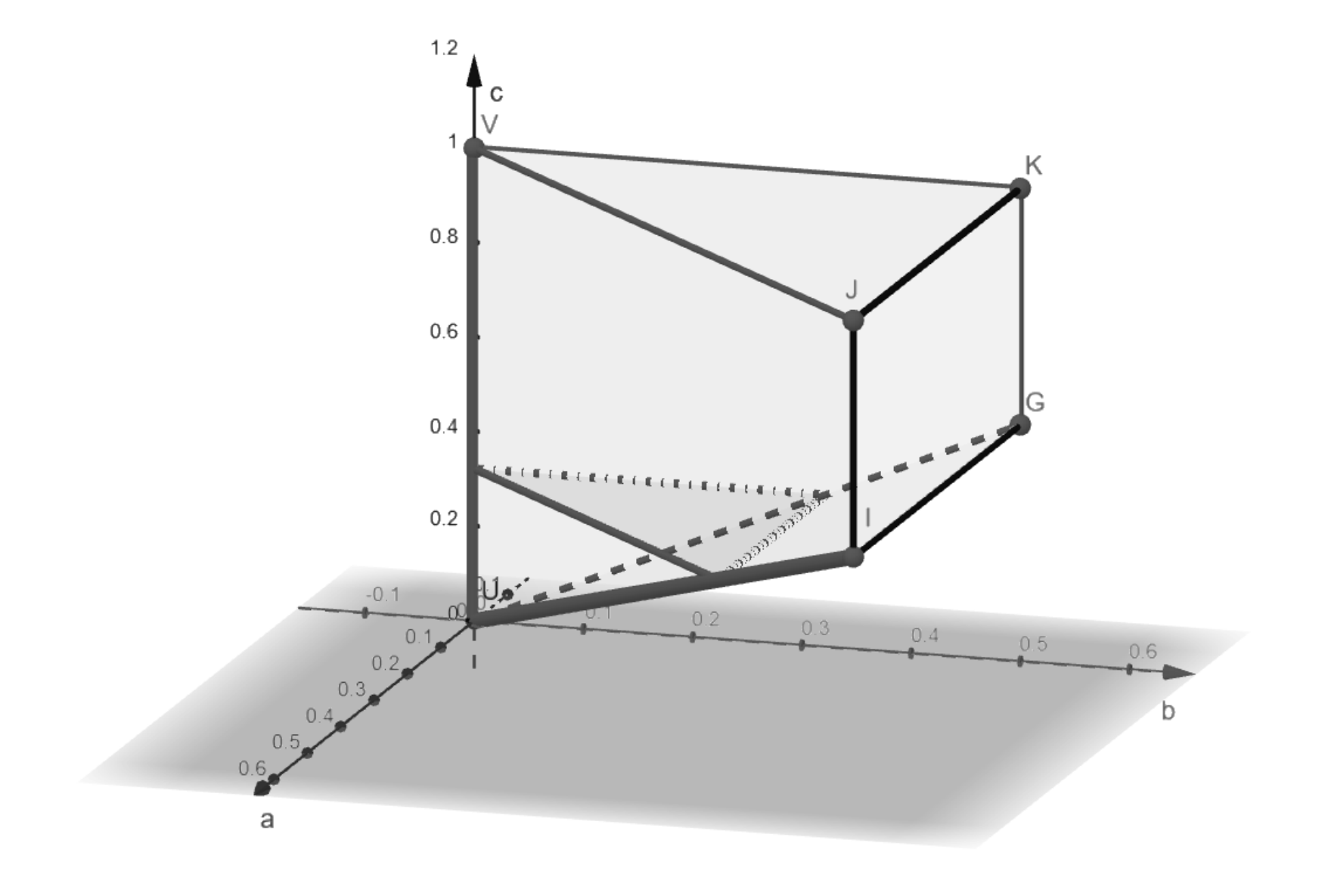}}}
\caption{The domain $S$ and its intersection with the plane $c=1/3.1$.}
\label{Yousrafig3D}
\end{figure}

\begin{nota}
  {\rm We denote by $\xi_1 \leq \xi_2 \leq \xi_3 \leq \xi_4$ the roots
    of $F'$ and by}

  $$z_1~=~\frac{a}{2}~\leq ~z_2~=~\frac{a+b}{2}~\leq ~z_3~
  =~\frac{b+c}{2}~\leq~z_4~=~\frac{c+1}{2}$$
  {\rm the midpoints of $F$. The same notation
    is used for the midpoints and roots of derivatives of other degree $5$
  hyperbolic polynomials encountered in the text.}
  \end{nota}

\begin{rem}\label{remmM}
  {\rm As $b\leq 1/2$ on $S$, for the minimal and maximal
    of the distances between consecutive
    midpoints one has}

  $$\begin{array}{cclclc}
    m&:=&\min (\frac{b}{2},\frac{c-a}{2},\frac{1-b}{2})&=&
    \min (\frac{b}{2},\frac{c-a}{2})&{\rm and}\\ \\
    M&:=&\max (\frac{b}{2},\frac{c-a}{2},\frac{1-b}{2})&=&
    \max (\frac{c-a}{2},\frac{1-b}{2})~.&\end{array}$$
  \end{rem}
  
The following example shows why it is not true that
for all points of the domain $S$ one has $L+$ or $R-$. Hence the proof of
non-realizability has to treat different parts of $S$ differently.

\begin{ex}\label{ex2possibilities}
  {\rm (1) For the polynomial $f_1:=x(x-0.5)^2(x-1)^2$
    (it corresponds to the vertex $J$, see (\ref{equvertices}) and
    Fig.~\ref{Yousrafig3D}),
    one obtains $m=M=0.25$.
    The roots of $f_1'$ are}

  $$\xi_1=0.129\ldots ~,~~~\, \xi_2=0.5~,~~~\, \xi_3=0.770\ldots ~,~~~\,
  \xi_4=1~.$$
{\rm Hence $\xi_2-\xi_1>M$ and one has $R-$. At the same time
    $\xi_4-\xi_3<0.23<m$, so one has $L-$ and $f_1$ realizes 
the case $L-R-$. For the polynomial 
$f_2:=x^2(x-0.5)(x-1)^2$ (see the vertex $K$ in (\ref{equvertices}) and
in Fig.~\ref{Yousrafig3D}),
one gets $m=0.25$ and $M=0.5$.
Numerical computation yields}

$$\xi_1=0~,~~~\, \xi_2=0.276\ldots~,~~~\, \xi_3=0.723\ldots~,~~~\, \xi_4=1~.$$
{\rm Thus one has $\tilde{m}=\xi_2-\xi_1>m$ and $\tilde{M}=\xi_3-\xi_2<M$.
  This is the case $L+R+$. By continuity, for all points of $S$ with distinct
    values of $0$, $a$, $b$, $c$, $1$ and sufficiently close to the ones
    corresponding to $f_1$ (resp. $f_2$),
    one has $L-R-$
    (resp. $L+R+$).
    \vspace{1mm}

    (2) We give also three numerical examples from the interior
    of the domain $S$:

    $$\begin{array}{ccllll}
      (a,b,c)&{\rm realizable~case}&~~~m&~~~\tilde{m}&~~M&~~\tilde{M}\\ \\
      (0.1,~0.49,~0.92)&L+R+&0.245&0.2559\ldots&0.41&0.3992\ldots\\ \\
      (0.2,~0.49,~0.92)&L+R-&0.245&0.2584\ldots&0.36&0.3635\ldots\\ \\
      (0.4,~0.49,~0.92)&L-R-&0.245&0.2396\ldots&0.26&0.3268\ldots\end{array}$$
      Hence along a segment of length $0.3$ one encounters all three cases
      different from $L-R+$, the moduli of the differences $\tilde{m}-m$
      and $\tilde{M}-M$ can drop to less than $0.011$, in certain cases to less
      than~$0.0036$. This should be compared to the fact that the roots
      remain always stretched over a segment of length~$1$.}

\end{ex}

\subsection{Moving the roots of hyperbolic polynomials}

The roots $a$, $b$ and $c$ being parameters, it would be useful to know how
the roots of $F'$ change when these parameters vary.

\begin{prop}\label{prop0}
  Suppose that a hyperbolic degree $d$ polynomial $F$ has roots
  $x_1<x_2<\cdots <x_k$ of multiplicities $m_1$, $m_2$, $\ldots$, $m_k$,
  $m_1+\cdots +m_k=d$. Suppose that one of them (say $x_j$) is shifted to
  its right (resp. left) without changing the order and the multiplicities
  of the roots:

  $$x_j\mapsto x_j\pm u~,~~~\, u>0~,~~~\, 
  u<\min_{j=1,\ldots ,k-1}(|x_j-x_{j+1}|)~.$$
  Then every root of the derivative
  $F'$ is shifted to the
  right (resp. left) and the sum of all these shifts equals $(d-1)m_ju/d$
  (resp. $-(d-1)m_ju/d$).
  \end{prop}

With regard to this proposition we remind a classical result of Vladimir Markov
(see \cite{Ma}, \cite[Lemma~1, Corollary~2]{Di} or \cite[Lemma~2.7.1]{Ri})
which states that every root of any derivative of a strictly hyperbolic
polynomial is an increasing function of each root of the
polynomial itself.

\begin{proof}
  The function $F'/F=\sum_{i=1}^km_i/(x-x_i)$ is decreasing on every interval

  $$(-\infty ,x_1)~,~~~\, 
  (x_1,x_2)~,~~~\, \ldots ~,~~~\,(x_{k-1},x_k)~,~~~\, (x_k,\infty )~.$$
  The function
  $1/(x-x_j-u)-1/(x-x_j)$ is positive-valued on $(-\infty ,x_j)$ and
  $(x_j+u,\infty )$. Hence every root of $F'$ is shifted to the right.
  (Similarly for the left shift.) If $F=x^d+b_{d-1}x^{d-1}+\cdots$, then
  after the shift one has

  $$F=x^d+(b_{d-1}-m_ju)x^{d-1}+\cdots ~~~\, {\rm and}~~~\, 
  F'=dx^d+(d-1)(b_{d-1}-m_ju)x^{d-1}+\cdots$$
  and the last statement
  of the proposition follows from Vieta's formulae.
  
\end{proof}

\subsection{Plan of the proof of Theorem~\protect\ref{tm1}}

The proof of Theorem~\ref{tm1} consists of an analytic and a numerical part.
The aim of the analytic part (see Section~\ref{secamput})
is to prove that on some subset $S_1$ of $S$
the case $L-R+$ is not realizable. The set $S_1$ contains all polynomials
having triple or quadruple roots. 
Avoiding them is necessary in order the
numerical part to be correctly defined. The set
$S_2:=\overline{S\setminus S_1}$ is the cylinder

$$S_2~:=~\{ ~(a,b,c)\in \mathbb{R}^3~
|~a\in [0,~b]~,~b\in [0.25,~0.5]~,~c\in [0.6,~1]~\}~$$
over the right trapezoid

$$\{ ~(a,b)\in \mathbb{R}^2~|~a\in [0,~b]~,~b\in [0.25,~0.5]~\}~.$$ 
We set $S_1=T_1\cup \cdots \cup T_6$, where the borders of the
sets $T_j$ are defined by
the border of the set $S$ and by one or two planes
whose equations do not depend on the variable $a$ 
(see Fig.~\ref{Yousrafig2D} in which we present the projections
of these sets in the plane $Obc$). We list here the sets $T_j$, the conditions
defining them and the statements in the text where they are mentioned:

$$\begin{array}{ccll}
  T_1&c\leq 1/3.1=0.322\ldots &{\rm Proposition~\protect\ref{prop3.1}}&{\rm
    Subsection~\protect\ref{subsecamput1}}\\ \\
  T_2&c\geq 3b,~b\in [0,0.25]&{\rm Corollary~\protect\ref{corc3b}}&{\rm
    Subsection~\protect\ref{subsecamput1}}\\ \\
  T_3&b\geq 0.1,~c\leq 0.5&{\rm Proposition~\protect\ref{prop0105}}&
  {\rm Subsection~\protect\ref{subsecamput2}}\\ \\
  T_4&b\in [1/6,0.25],~c\geq 0.5&
  {\rm Proposition~\protect\ref{prop1614}}&
  {\rm Subsection~\protect\ref{subsecamput2}}\\ \\
  T_5&b\in [0.35,0.5],~c\in [0.5,0.6]&
  {\rm part~(2)~of~Proposition~\protect\ref{prop0506R-}}&
  {\rm Subsection~\protect\ref{subsecamput3}}\\ \\ 
  T_6&b\in [0.25,0.35],~c\in [0.5,0.6]&
  {\rm part~(1)~of~Proposition~\protect\ref{prop0506R-}}&
  {\rm Subsection~\protect\ref{subsecamput3}}~.
  \end{array}$$
In Fig.~\ref{Yousrafig2D} the projection in the $Obc$-plane
of the set $S_2$ is given in
black, the ones of $T_3$, $T_4$ and $T_5$ are shaded in grey. The projection
of $T_6$ (between the ones of $S_2$, $T_5$, $T_4$ and $T_3$) is in white.
\begin{figure}[htbp]
\centerline{\hbox{\includegraphics[scale=0.6]{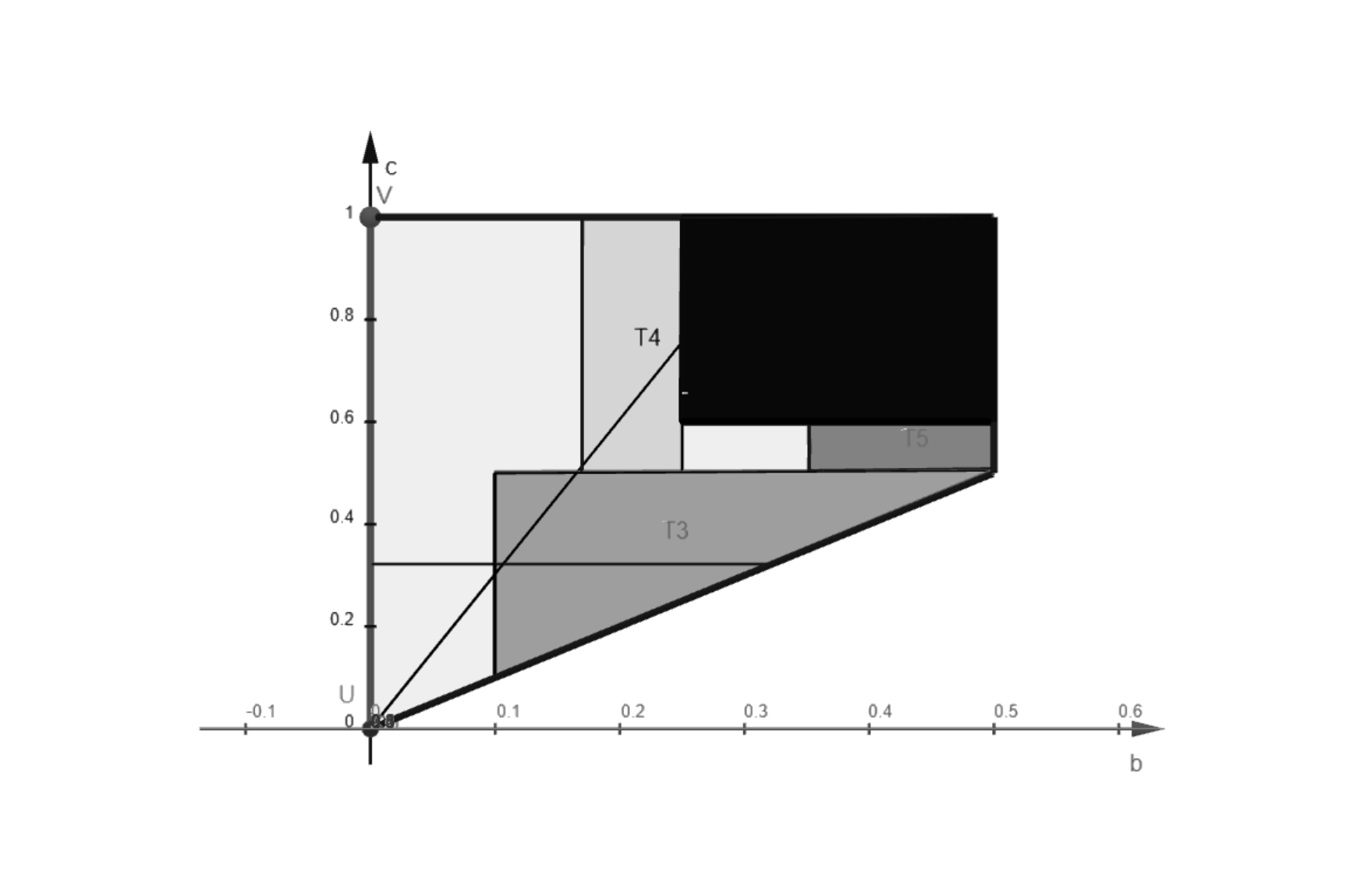}}}
\caption{The domains $T_j$ and $S_2$.}
\label{Yousrafig2D}
\end{figure}

The numerical part of the proof of Theorem~\ref{tm1}
is explained in Section~\ref{secnumer}. A numerical method
is preferred to an analytic one for dealing with the set $S_2$, because
in $S_2$ the roots are distributed ``more regularly''
(there are no triple or quadruple roots). Thus continuing by
analytic methods would require to subdivide the domain $S_2$ into too many small
parts. This would render the proof of the nonrealizability of the case $L-R+$ 
harder and lengthier; see part ~(2) of Example~\ref{ex2possibilities}.

 \begin{rems}
   {\rm (1) One can notice that the sets $T_1$ and $T_2$ 
     contain the vertex $U$
    of the set~$S$ which defines the polynomial $x^4(x-1)$
    having a quadruple root. This is the only such polynomial in~$S$. 
    \vspace{1mm}

    (2) When defining the set $T_2$ (see Corollary~\ref{corc3b}) one can drop  
    the condition $b\in [0,0.25]$. It is added here in order the set $S_2$ to
    have the relatively simple form of a cylinder over a right trapezoid. This
    simplifies the definition of the numerical part
    of the proof of Theorem~\ref{tm1}.
    \vspace{1mm}
    
    (3) One has $UV\subset T_1\cap T_2$,
    where the edge
  $UV$ contains all polynomials of $S$ with 
    a triple root at $0$ and two simple positive roots
    or a double positive root. The edge $UI$ of $S$
    (see (\ref{equvertices}) and Fig.~\ref{Yousrafig3D})
    contains all polynomials
    having a triple root in the
    middle. This edge belongs to the set $T_1\cup T_3$. The union
    of the closures of these two edges contains all points of $S$ defining
    polynomials with triple or quadruple roots; it contains
    no point of the set $S_2$.
    \vspace{1mm}

    (4) The intersection line of the planes $c=3b$ and $c=1/3.1$
    defining parts of the borders of the sets $T_1$ and $T_2$ is of the form
    $b=0.107\ldots$, $c=1/3.1$. It belongs to the set
    $\overline{T_3\setminus T_1}$. The line segment $c=3b=0.5$,
    $0\leq a\leq b$ belongs to the borders
    of the sets $T_2$, $T_3$ and $T_4$, see Fig.~\ref{Yousrafig2D}.
    Hence the union $T_1\cup \cdots \cup T_6$
    covers the whole of the set $\overline{S\setminus S_2}$.
    \vspace{1mm}
    
    (5) The set $T_4$ is added to $S_1$ in order the latter's border
    to consist only of 
    horizontal or vertical faces and of part of the plane $a=b$. Horizontal
    or vertical are the planes parallel to the $Oab$- and $Oac$-planes
    respectively. In this sense the plane $c=3b$ defining
    part of the border of the set $T_2$ is neither horizontal nor vertical.
    \vspace{1mm}

    (6) The set $T_5$ is added to $S_1$ in order the point $I$
    (see (\ref{equvertices}) and Fig.~\ref{Yousrafig3D}) 
    defining the polynomial $x(x-0.5)^3(x-1)$ with a triple root
    not to belong to the border of $S_2$. The set $T_6$ is added in order to
    obtain the simple form of $S_2$ mentioned in part~(2) of these remarks.  
    \vspace{1mm}

    (7) The polynomials $f_1$ and $f_2$ of Example~\ref{ex2possibilities}
    define points of
    the set~$S_2\setminus \overline{S_1}$.
  }
 \end{rems}

\section{Analytic part of the proof of Theorem~\protect\ref{tm1}
  \protect\label{secamput}}

\subsection{Deleting polynomials with triple or quadruple roots
  \protect\label{subsecamput1}}

In this subsection we prove that the case $L-R+$ fails
on the sets $T_1$ (see Proposition~\ref{prop3.1})
and $T_2$ (see Corollary~\ref{corc3b}).
We consider the polynomial

$$P:=x(x-\alpha)(x-\beta)(x-1)(x-A)~,~~~\, {\rm where}~~~\,
0\leq \alpha \leq \beta \leq 1 \leq  A~.$$
We remind that we denote by $\xi_1\leq \xi_2\leq \xi_3 \leq \xi_4$
the roots of~$P'$ and by 
$z_1\leq z_2\leq z_3 \leq z_4$ the midpoints of~$P$. The change of
variable $x\mapsto Ax$ transforms the polynomial $P$ into
$A^5F$ (see (\ref{equF})) with $a=\alpha /A$, $b=\beta /A$, $c=1/A$.

\begin{lm}
  If $A\geq 2$, then the distance
  $z_4-z_3$ is the largest of the distances between consecutive midpoints.
  One has $\frac{A}{2}\geq z_4-z_3\geq \frac{A-1}{2}\geq \frac{1}{2}$.
\end{lm}

\begin{proof}
 Indeed, one has
 $\frac{A}{2}=\frac{1+A}{2}-\frac{1}{2}\geq z_4-z_3\geq \frac{1+A}{2}-1=
 \frac{A-1}{2}\geq \frac{1}{2}$ while
 $z_3-z_2$ and 
 $z_2-z_1$ are $\leq \frac{1}{2}$.
\end{proof}

\begin{prop}\label{prop3.1}
  For $A\geq 3.1$, one has
  $\xi_4-\xi_3>\xi_4-1>\frac{A}{2}\geq z_4-z_3$
  and hence the polynomial $P$ does not realize the case~$R+$
  (hence not the case $L-R+$ either). Thus the case $L-R+$ is not realizable
  by a polynomial $F$ for $c\leq 1/3.1=0.3225806452\ldots$.
\end{prop}

\begin{proof}
  When the root $\beta$ decreases, then the difference $z_4-z_3$ increases
  while $\xi_4-1$ decreases, see Proposition~\ref{prop0}. The same is true
  when both $\alpha$ and $\beta$ decrease. Hence it suffices to prove
  the proposition for $\alpha =\beta =0$. For $P=x^3(x-1)(x-A)$,
  one obtains

  $$\xi_4=(2A+2+\sqrt{4A^2-7A+4})/5~.$$
  The quantity
$\xi_4-1-\frac{A}{2}$ is increasing for $A>1$; it is equal to $0$ when
$A=\frac{4}{3}+\frac{2}{3}\sqrt{7}=3.09716..$.
When $A\geq 3.1$, one has
$\xi_4-1>\frac{A}{2}\geq z_4-z_3$ and the polynomial does not
realize the case~$R_+$.

\end{proof}

We treat now the case of the open edge $UV$ of the simplex $S$
on which there is
a triple root at $0$ and two simple positive roots. Our aim is to include the
edge in a subdomain of $S$ on which the case $L-$ fails. We prefer a
new parametrization, so we set

$$Q:=x(x-r)(x-1)(x-t_1)(x-t_2)~,~~~\, r\in [0,1]~,~~~\, 1\leq t_1\leq t_2~.$$ 
We need the following proposition:

\begin{prop}\label{propL+}
  For $r\in [0,1]$ and $3\leq t_1\leq t_2$, the polynomial $Q$ realizes the case
  $L+$.
\end{prop}

\begin{cor}\label{corc3b}
  The case $L-R+$ is not realizable by a polynomial $F$
  for $c\geq 3b$.
\end{cor}

Indeed, the change of variable $x\mapsto t_2x$ transforms the polynomial $Q$
into $t_2^5F$ (see (\ref{equF})) with $a=r/t_2$, $b=1/t_2$, $c=t_1/t_2$.

\begin{proof}[Proof of Proposition~\ref{propL+}]
One sets

$$H:=Q'/Q=1/x+1/(x-r)+1/(x-1)+1/(x-t_1)+1/(x-t_2)~.$$ 
This function is decreasing on each of the intervals

$$(-\infty ,0)~,~(0,r)~,~(r,1)~,~(1,t_1)~,~(t_1,t_2)~~~\, {\rm and}~~~\,
(t_2,\infty )~.$$
\begin{lm}\label{lm1}
  One has $\xi_2-\xi_1>1/2$.
\end{lm}

\begin{proof}
The quantities $H(r/2)$ and $H((r+1)/2)$ increase as $t_1$ or $t_2$ increases.
For $t_1=t_2=3$, one has

$$\begin{array}{ccl}H(r/2)&=&2/r-2/r-1/(1-r/2)-1/(t_1-r/2)-1/(t_2-r/2)\\ \\
  &=&-1/(1-r/2)-1/(t_1-r/2)-1/(t_2-r/2)<0\end{array}$$
while for $H((r+1)/2)$, setting $\rho :=r+1$, one obtains

$$\begin{array}{ccl}
  H(\rho /2)&=&2/\rho +2/(1-r)-2/(1-r)-2/(2t_1-\rho )-2/(2t_2-\rho )\\ \\
  &=&2/\rho -2/(2t_1-\rho )-2/(2t_2-\rho )~.\end{array}$$
As $r$ increases, this quantity decreases. For $r=1$ and $t_1=t_2=3$,
it equals $0$, so $H((r+1)/2)\geq 0$ with equality only for $r=1$, $t_1=t_2=3$.
Hence $\xi_1<r/2$ and $\xi_2\geq (r+1)/2$.
\end{proof}

Next, we show that the distance between $\xi_1$ and $\xi_2$ is the smallest
of the distances between consecutive roots of $Q'$.

\begin{lm}\label{lm2}
  One has $\xi_3-\xi_2>\xi_2-\xi_1$.
\end{lm}

\begin{proof}
To prove this we notice that

$$\begin{array}{ccl}H((t_1+1)/2)&=&
  2(\frac{1}{t_1+1}+\frac{1}{t_1+1-2r}+\frac{1}{t_1-1}-\frac{1}{t_1-1}
  -\frac{1}{2t_2-t_1-1})\\ \\
  &=&2(\frac{1}{t_1+1}+\frac{1}{t_1+1-2r}-\frac{1}{2t_2-t_1-1})\\ \\ &\geq &
  2(\frac{1}{t_1+1}+\frac{1}{t_1+1-2r}-\frac{1}{t_1-1})=
  \frac{2(t_1^2-2t_1+4r-3)}{(t_1+1)(t_1+1-2r)(t_1-1)}~.\end{array}$$
The numerator is $\geq 0$ for $t_1\geq 3$ and $r\geq 0$, so
$H((t_1+1)/2)\geq 0$, $\xi_3\geq (t_1+1)/2\geq 2$ and $\xi_3-\xi_2\geq 1$
whereas $\xi_2-\xi_1<1$.
\end{proof}

\begin{lm}\label{lm3}
  One has $\xi_4-\xi_3>\xi_2-\xi_1$.
\end{lm}

\begin{proof}
For fixed $r$, the quantity
$\xi_1$ is  
minimal when $t_1=t_2=3$. Suppose that $t_1=t_2=3$ and rescale the $x$-axis:
$x:=ry$, $y\in (0,1)$. Then one obtains the equation

$$1/y+1/(y-1)+1/(y-1/r)+2/(y-3/r)=0$$
whose solution $\xi_1'\in (0,1)$ is a decreasing function in~$r$. For
$r\rightarrow 0$ (resp. $r\rightarrow 1$), one has $\xi_1'\rightarrow 1/2$
(resp. $\xi_1'\rightarrow 0.283484861\ldots =:\kappa$). Thus
$\xi_1/r>\kappa$. 

For fixed $r$, the quantity
$\xi_2$ tends to its supremum as $t_1,t_2\rightarrow \infty$. This supremum is
the larger of the two solutions to the equation

$$1/x+1/(x-r)+1/(x-1)=0~.$$
It equals $\tilde{\xi}:=(r+1+\sqrt{r^2-r+1})/3$. 
%
Hence

$$\xi_2-\xi_1~\leq ~\tilde{f}~:=~\tilde{\xi}-\kappa r~.$$
One finds directly that $\tilde{f}''=1/(4(r^2-r+1)^{3/2})>0$. Hence
the derivative $\tilde{f}'$
is an increasing function. As $\tilde{f}'(0)=-0.11\ldots <0$ and
$\tilde{f}'(1)=0.21\ldots >0$, the function $\tilde{f}$ has a single critical
point on $[0,1]$, which is a minimum. Thus for $r\in [0,1]$,

$$\xi_2-\xi_1~\leq ~\min (\tilde{f}(0),~\tilde{f}(1))~=~
\min (2/3,~0.7165151389\ldots )<0.717~.$$

Now we find a lower bound for $\xi_4-\xi_3$. We observe first that
$\xi_4>z_4=(t_1+t_2)/2$, because $H(z_4)>0$. Set $t:=t_2-t_1$, so

$$H((t_1+1)/2)=2\left( 1/(t_1+1)+1/(t_1+1-2r)-1/(t_1+2t-1)\right) ~.$$
For $t_1$ fixed
and for $t\geq 0.6$,
one obtains a lower bound for $t_1-\xi_3$ by letting $t_2$ (hence $t$)
tend to $\infty$. Then one decreases further $t_1-\xi_3$ by setting
$r:=1$ and finally by replacing the term $1/x$ in $H$ by $1/(x-1)$. Thus
finally one replaces $H$ by

$$H_*:=3/(x-1)+1/(x-t_1)~.$$
The zero of $H_*$
equals $\xi_3^*:=(3t_1+1)/4$ and the difference $t_1-\xi_3^*=(t_1-1)/4$
is minimal for $t_1=3$. Hence for $t\geq 0.6$, one has
$t_1-\xi_3\geq 1/2$ and

$$\begin{array}{ccccccc}\xi_4-\xi_3&=&
  (\xi_4-t_1)+(t_1-\xi_3)&>&z_4-t_1+t_1-\xi_3\\ \\
  &=&(t_2-t_1)/2+t_1-\xi_3&=&t/2+t_1-\xi_3\\ \\ &\geq &0.3+1/2&=&0.8&>&0.717~.
  \end{array}$$

Suppose now that $t\leq 0.6$. For $t_1$ and $t$ fixed, the quantity $t_1-\xi_3$
is minimal when $r=1$. Suppose that $r=1$. Then for fixed $t$, 
the difference $t_1-\xi_3$ is minimal when $t_1=3$.

Indeed, if the sum of the roots of $Q$ equals $h>0$, then the one of $Q'$
equals $4h/5$. For fixed $t_2$, if $t_1$ decreases by $\Delta t_1$, then
all roots of $Q'$ decrease and $\xi_3$ decreases by not more than
$4\Delta t_1/5$. Hence as $t_1$ decreases, the difference $t_1-\xi_3$ also
decreases. Thus for $r=1$ and for $t_2$ fixed, the difference $t_1-\xi_3$
is minimal for $t_1=3$.

If $r=1$ and $t_1=3$, then
$t_1-\xi_3$ is minimal for $t=0.6$. The zero of the function

$$1/x+2/(x-1)+1/(x-3)+1/(x-3.6)$$
which is in the interval $(1,3)$ equals $\lambda :=2.242184744\ldots$
and the difference $3-\lambda =0.75\ldots$ is $>0.717$, so one has
$\xi_4-\xi_3>0.717>\xi_2-\xi_1$.

\end{proof}

Lemmas~\ref{lm1}, \ref{lm2} and \ref{lm3} imply that
in the conditions of Proposition~\ref{propL+} one has $L+$ and not $L-$.

\end{proof}

\subsection{Further decreasing of the domain $S$\protect\label{subsecamput2}}

In the present subsection we prove Proposition~\ref{prop0105}
(preceded by Remark~\ref{remb0}) and
Proposition~\ref{prop1614} which allow to further delete from $S$ parts of it
(namely, the sets $T_3$ and $T_4$) on which
one does not have $L-R+$. We precede Proposition~\ref{prop0105}
by some lemmas which
are used in its proof.

\begin{lm}
    For $c\leq 0.5$, one has $M:=\max (z_2-z_1,z_3-z_2,z_4-z_3)=z_4-z_3$.
  \end{lm}

  \begin{proof}
    Indeed, $z_4-z_3=(1-b)/2\geq 0.25$ while $z_3-z_2=(c-a)/2\leq c/2\leq 0.25$
    and $z_2-z_1=b/2\leq 0.25$.
    \end{proof}

\begin{lm}\label{lmkappa0}
  (1) For $c\leq 0.6$,
  one has $\xi_4-z_4\geq \kappa_0:=0.0627105746\ldots >0.05$.

  (2) For $c\leq 0.5$,
  one has $\xi_4-z_4\geq \kappa_1:=0.0949489742\ldots$.
    \end{lm}

    \begin{proof}
    Fix $c\leq 0.6$. The difference
    $\xi_4-z_4$ is minimal when $a=b=0$.
    (This follows from Proposition~\ref{prop0} -- when $a$ and $b$ decrease,
    then $\xi_4$ decreases.) So we consider the polynomial
    $P_0:=x^3(x-c)(x-1)$. The non-zero roots of $P_0'$ are 

    $$ \xi_{\pm}:=(2(c+1)\pm \sqrt{4c^2-7c+4})/5~.$$
    The difference $\xi_+-z_4=\xi_+-(1+c)/2$ is a decreasing function in $c$
    for $c\in [0,1]$. Indeed, its derivative is $E:=-(A+C)/10A$, where 

    $$A:=\sqrt{4c^2-7c+4}~~~\, {\rm and}~~~\, C:=-8c+7~,~~~\, {\rm so}$$
    $$A+C=\sqrt{4(c-1)^2+c}+7(1-c)-c\geq \sqrt{c}-c\geq 0~,$$
    with equality only for $c=1$. Thus 
$E\leq 0$ and
    the difference $\xi_+-z_4$ is decreasing. For $c=0.6$, this difference is
    $\kappa_0:=0.0627105746\ldots >0.05$. For $c=0.5$, the difference
    equals $\kappa_1:=0.0949489742\ldots$.
    \end{proof}

    \begin{lm}\label{lmc06c-b025}
      (1) Suppose that $c\leq 0.6$ and $c-b\leq 0.25$.
      Then $\xi_4-\xi_3>z_4-z_3$.

      (2) The same is true for $c\leq 0.5$ and $c-b\leq 0.379$.
      \end{lm}

    \begin{proof}
      Part (1). For $c-b\leq 0.25$ fixed, the difference $\xi_3-z_3$ is maximal
      when $a=b=0$. 
      Indeed, one can apply Proposition~\ref{prop0} to the polynomial $F$,
      see (\ref{equF}). For fixed $b$ and $c$,
      one increases the roots
      $0$ and $a$ of $F$ until they become equal to $b$. Hence the root
      $\xi_3$ increases. Now we change also the position of the root~$1$.
      We increase this root until it becomes
      equal to $1+b$. Again the root $\xi_3$ increases. All these changes do
      affect the quantity $z_3=(b+c)/2$. Then one makes the
      shift $x\mapsto x-b$. The roots of $F$
      become equal to $0$, $0$, $0$, $c$
      and $1$, so $F=P_0=x^3(x-c)(x-1)$.
      The latter shift does not change the difference $\xi_3-z_3$.
      Hence this difference is maximal for $a=b=0$.
      So consider the polynomial

      $$P_0'=3x^2(x-c)(x-1)+x^3(x-1)+x^3(x-c)=x^2(5x^2-4(c+1)x+3c)~.$$
      As $P_0'(3c/4)=-27c^4/256<0$, one concludes that $\xi_3<3c/4$. Recall
      that $b=0$, so $c-b=c$, and that the constants $\kappa_0$ and $\kappa_1$
      were introduced by Lemma~\ref{lmkappa0}. For
      $c\leq 0.25$, one obtains

      $$\xi_3-z_3=\xi_3-c/2<c/4\leq 0.0625<\kappa_0\leq \xi_4-z_4~,$$
      so $\xi_4-\xi_3>z_4-z_3$.

      Part (2). By complete analogy one concludes that $\xi_3<3c/4$.
      For $c\leq 0.379$, one gets

      $$\xi_3-z_3=\xi_3-c/2<c/4\leq 0.379/4=0.09475<\kappa_1\leq \xi_4-z_4$$
      and again $\xi_4-\xi_3>z_4-z_3$.
      
      \end{proof}

    \begin{rem}\label{remb0}
      {\rm We showed already that for $c\leq 1/3.1$ and for $c\geq 3b$
        (see Proposition~\ref{prop3.1} and Corollary~\ref{corc3b}),
        one does not have $L-R+$. The intersection of the planes
        $c=c_0:=1/3.1=0.3225806452\ldots$ and $c=3b$
        (both parallel to the $a$-axis)
        is the straight line $b=b_0:=0.1075268817\ldots$, $c=c_0$. Hence
        for $b\leq b_0$, the case $L-R+$ is not realizable.}
      \end{rem}

    \begin{prop}\label{prop0105}
For $b\geq 0.1$ and $c\leq 0.5$, the case $R+$ fails.
      \end{prop}

\begin{proof}
  We set $a:=0$ and we compare the roots of $P_*'$ and $P_{\dagger}'$,
  where

  $$P_*:=x^2(x-b)(x-c)(x-1)~~~\, {\rm and}~~~\, P_{\dagger}:=x^2(x-b)(x-c)~.$$
  Their roots are the same as the ones of $P_*'/P_*$ and
  $P_{\dagger}'/P_{\dagger}$. We denote them by
  $\xi_{1,*}<\cdots <\xi_{4,*}$ and $\xi_{1,\dagger}<\cdots <\xi_{3,\dagger}$; when
  necessary, we omit their second indices.
  The difference $P_*'/P_*-P_{\dagger}'/P_{\dagger}$ is
  just the term $1/(x-1)$ which is negative on $[b,c]$. Therefore for fixed $b$
  and $c$, for the
  roots $\xi_3$ of $P_*'$ and $P_{\dagger}'$, one has $\xi_{3,*}<\xi_{3,\dagger}$.

  Suppose that $c-b$ is fixed and $b\geq 0.1$. Then the difference
  $\xi_{3,\dagger}-z_3=\xi_{3,\dagger}-(b+c)/2$ is maximal for $b=0.1$. Indeed,
  set $x:=y+b$. Then in

  $$(P_{\dagger}'/P_{\dagger})(y+b)=2/(y+b)+1/y-1/(y+(b-c))$$
  only the first term varies with $b$. This term is maximal (for any
  $y\in [0,c-b]$ fixed) for $b=0.1$. Thus one obtains as upper bound for
  $\xi_3-z_3$ the value of $\xi_{3,\dagger}-(0.1+c)/2$. One needs to consider
  only the case $c-b\geq 0.379$ (see Lemma~\ref{lmc06c-b025}),
  i.~e. $c\geq 0.479$
  hence $c\in [0.479,0.5]$.

  One finds numerically that for $c=0.479$, $0.484$, $0.489$, $0.494$ and $0.5$,
  the respective values of $\xi_{3,\dagger}-(0.1+c)/2$ are

  $$0.07991\ldots ~,~~~0.08114\ldots~,~~~0.08237\ldots ~,~~~0.08237\ldots
  ~~~{\rm and}~~~
  0.08507\ldots ~.$$
  Therefore for these values of $c$, one has $\xi_4-\xi_3>z_4-z_3$,
  because $\xi_4-z_4\geq \kappa_1$ (see Lemma~\ref{lmkappa0})
  while $\xi_3-z_3\leq 0.08507\ldots$. To
  obtain an upper bound for $\xi_3-z_3$ for the intermediate values of $c$
  one observes that for $b=0.1$, as $c$ increases, then the root $\xi_3$
  also increases. The differences between any two consecutive of the
  above five values of $c$ are $\leq 0.006$. By Proposition~\ref{prop0}
  the value of $\xi_{3,\dagger}$ increases by less than $0.006$, so the one of
  $\xi_{3,\dagger}-(0.1+c)/2$ also increases by less than $0.006$.

  That's why for the intermediate values of $c$ one can choose as upper bound
  for $\xi_3-z_3$ the quantity $0.08507\ldots +0.006=0.09107\ldots <\kappa_1$
  and again $\xi_4-\xi_3>z_4-z_3$. This proves the proposition.

\end{proof}

 \begin{prop}\label{prop1614}
    Suppose that $b\in [1/6,1/4]$ and $c\geq 1/2$. Then one has $L+$.
    \end{prop}

\begin{proof}
  The proof of the proposition boils down to the proofs of the following
  three lemmas:

  \begin{lm}
    For $b\in [1/6,1/4]$ and $c\geq 1/2$, one has $\xi_2-\xi_1>z_2-z_1~(*)$.
  \end{lm}

  \begin{proof}
    Recall that $\xi_1<z_1$ (Lemma~\ref{lmbasic}).
    We show that for $c\geq 5/8$, one has
  $\xi_2\geq z_2$, so $\xi_2-\xi_1>z_2-z_1$. Indeed, set $\gamma :=z_2=(a+b)/2$.
  Then

  $$\begin{array}{ccl}
    (P'/P)(\gamma )&=&1/\gamma -1/(c-\gamma )-1/(1-\gamma )\\ \\ &=&
  (3\gamma ^2-2(1+c)\gamma +c)/\gamma (c-\gamma )(1-\gamma )~.\end{array}$$
  The roots of the numerator are

  $$r_{\pm}~:=~((1+c)\pm \sqrt{c^2-c+1})/3~.$$
  Hence $(P'/P)(\gamma )\geq 0$ exactly when either $\gamma \geq r_+$ or
  $\gamma \leq r_-$. As $\gamma \leq 1/4$ and $3\gamma \leq 1+c<3r_+$,
  the inequality 
  $\gamma \geq r_+$ is impossible. The inequality $1/4\leq r_-$
  is equivalent to $c\geq 5/8$ (to be checked directly). Thus for
  $c\in [5/8,1]$, the inequality $(*)$ holds true.

  We suppose from now on that $c\in [1/2,5/8]$. We show first
  that one has $\xi_2\geq z_2=\gamma$ also for $b\in [1/6,0.21]$.
  Indeed, $r_-$ is an increasing function in~$c$. For $c=1/2$, it equals
  $0.211\ldots >0.21$. Hence for $0\leq a\leq b\leq 0.21$, one obtains
  $\gamma \leq 0.21<r_-$ and $(P'/P)(\gamma )>0$.

  Next, for $a\in [0,0.17]$ and $b\in [0.21,0.25]$, one has $\gamma \leq 0.21$
  and again $(P'/P)(\gamma )>0$ and $(*)$ holds true.

  So now we suppose that $b\in [0.21,0.25]$ and $a\in [0.17,b]$.
  The right of inequalities (\ref{equAn}) with $d=5$ and $j=2$
  implies

  $$\begin{array}{l}
    \xi_2-a<2(b-a)/3~,~~~\, {\rm i.~e.}~~~\, \xi_2<(2b+a)/3<(3b+a)/4~,\\ \\ 
  {\rm so}~~~\, \xi_2-z_2=\xi_2-(a+b)/2<(b-a)/4~.\end{array}$$
  One can apply the left of
  inequalities (\ref{equAn}) with $d=5$ and $j=2$ to conclude that

  $$\begin{array}{l}
    \xi_2-a>(b-a)/4~,~~~\, {\rm i.~e.}~~~\, \xi_2>(b+3a)/4~~~\, {\rm and}~~~\,
    z_2-\xi_2<(b-a)/4~~~\, 
       {\rm hence}\\ \\ 
  |\xi_2-z_2|<(b-a)/4\leq (0.25-0.17)/4=0.02~.\end{array}$$
  We consider
  first the case $a\in [0.17,0.205]$. 
  For $a$ fixed,
  the difference $z_1-\xi_1$ is minimal for $b=0.25$ and $c=5/8$.
  We list the values of this difference (computed up to the third decimal)
  for $a=0.17+k\times 0.005$,
  $0\leq k\leq 7$:

  $$\begin{array}{ccllllllll}
    a&&0.17&0.175&0.18&0.185&0.19&0.195&0.2&0.205\\ \\
    z_1-\xi_1&&0.026&0.027&0.029&0.030&0.032&
    0.034&0.035&0.037 ~.\end{array}$$
By Proposition~\ref{prop0} when $a$ increases by $a^{\dagger}$,
$a^{\dagger}\in [0,b-a]$, $\xi_1$ increases by not more than $4a^{\dagger}/5$,
$z_1$ increases by exactly $a^{\dagger}/2$, so $z_1-\xi_1$
decreases by not more than $(4/5-1/2)a^{\dagger}=3a^{\dagger}/10$.
In the above table the step equals $0.005$, so between two values of $a$ of
the table the value of $z_1-\xi_1$ decreases by not more than $0.0015$.

Hence for $a\in [0.17,0.205]$, one has $z_1-\xi_1\geq 0.024$. At the same time
$|\xi_2-z_2|<0.02$, so

$$z_1-\xi_1-(z_2-\xi_2)>0.024-0.02=0.004>0$$
which implies (*).

Suppose now that $b\in [0.21,0.25]$, $a\in [0.205,b]$. We apply the above
reasoning with $a_0=0.205$, $a^{\dagger}\leq 0.25-0.205=0.045$ and
$(z_1-\xi_1)|_{a=a_0}\geq 0.037$ to conclude that

$$(z_1-\xi_1)|_{a=a_0+a^{\dagger}}\geq 0.037-0.045\times 0.3=0.0235>0.02~,$$
so again $z_1-\xi_1-(z_2-\xi_2)>0$ and (*) holds true.

  \end{proof}
  
\begin{lm}
    For $b\in [1/6,1/4]$ and $c\geq 1/2$, one has $\xi_3-\xi_2>z_2-z_1$.
\end{lm}

\begin{proof}
  For fixed $b$ and $c$, the root $\xi_3$ is minimal when $a=0$, see
  Proposition~\ref{prop0}. We show that in this case one has
  $\xi_3>z_3=(b+c)/2$.
  Indeed, set

  $$P_{\diamond}:=2/x+1/(x-b)+1/(x-c)+1/(x-1)~.$$
  Then $P_{\diamond}((b+c)/2)=2(4-3b-3c)/(b+c)(2-b-c)$. For $b\in [1/6,1/4]$
  and $c\in [1/2,1]$, all factors are positive. Thus $\xi_3>z_3$. Hence

  $$\xi_3-\xi_2>z_3-b=(c-b)/2>((1/2)-(1/4))/2=1/8\geq b/2=z_2-z_1~.$$
  \end{proof}

  \begin{lm}
    For $b\in [1/6,1/4]$ and $c\geq 1/2$, one has $\xi_4-\xi_3>z_2-z_1$.
  \end{lm}

  \begin{proof}
    For $b$ and $c$ fixed, the difference $z_3-\xi_3$ is the smallest
    possible when $a=b$, see Proposition~\ref{prop0}. So we consider the
    function

    $$P_{**}:=1/x+2/(x-b)+1/(x-c)+1/(x-1)~.$$
    One finds directly that

    $$\begin{array}{ccl}P_{**}((7c+3b)/10)&=&
      50B/(21(10-7c-3b)(c-b)(3b+7c))~,~~~\, {\rm where}\\ \\ 
    B&=&27b^2+42bc-49c^2-48b+28c~.\end{array}$$
    The polynomial $B$ takes only negative values. Indeed,
    $\partial B/\partial c=42b-98c+28$, with $42b\leq 10.5$ and $98c\geq 49$,
    so $\partial B/\partial c<0$. One finds that

    $$B_0:=B|_{c=1/2}=27b^2-27b+7/4~,~~~\, {\rm where}~~~\,
    \partial B_0/\partial b=27(2b-1)$$
    which is negative for $b\in [1/6,1/4]$. Hence $B_0\leq B_0|_{b=1/6}=-2<0$.
    This implies $\xi_3<(7c+3b)/10$.

    As $\xi_4>z_4=(1+c)/2$, one can minorize the quantity
    $(\xi_4-\xi_3)-(z_2-z_1)$ by

    $$((1+c)/2-(7c+3b)/10)-b/2=0.5-0.2c-0.3b\geq 0.5-0.2-0.3\times 0.25>0$$
    from which the lemma follows.
    
    \end{proof}
  
\end{proof}

\subsection{The domain $S_2$\protect\label{subsecamput3}}

In this subsection we prove Proposition~\ref{prop0506R-}. It defines two
more sets to be deleted from $S$ (these are $T_6$ and $T_5$)
after which one obtains the set $S_2$.

\begin{prop}\label{prop0506R-}
  (1) For $a\in [0,b]$, $b\in [0.25,0.35]$ and $c\in [0.5,0.6]$, one has $R-$.

  (2) For $a\in [0,b]$, $b\in [0.35,0.5]$ and $c\in [0.5,0.6]$,
  one has $R-$.
\end{prop}

\begin{proof}
  We prove part (2) first. We use Proposition~\ref{prop0} to find a lower bound
  for $\xi_4-\xi_3$ in the domain

  $$\Gamma_1~:=~
  \{ ~a\in [0,b]~,~~~\, b\in [0.35,0.4]~,~~~\, c\in [0.5,0.6]~\} ~.$$
  The minimal value of $\xi_4$ (resp. the maximal value of $\xi_3$)
  in this domain is attained for $(a,b,c)=(0,0.35,0.5)$ (resp.
  $(0.4,0.4,0.6)$). These values are $\xi_{4,1}:=0.8587868945\ldots$
  and $\xi_{3,1}:=0.5325539228\ldots$, so
  $\tilde{M}\geq \xi_{4,1}-\xi_{3,1}=0.326\ldots$. At the same time
  for $(a,b,c)\in \Gamma_1$, one has $M\leq 0.325$,
  see Remark~\ref{remmM}. Hence $\tilde{M}>M$ in~$\Gamma_1$. 

  Next, we consider the domain

  $$\Gamma_2~:=~\{ ~a\in [0,b]~,~~~\, b\in [0.4,0.48]~,~~~\, 
  c\in [0.5,0.6]~\} ~.$$
  In a similar way we get $\xi_{4,2}=0.8620109877\ldots$
  and $\xi_{3,2}:=0.5594818435\ldots$ (obtained for
  $(a,b,c)=(0,0.4,0.5)$ and $(0.48,0.48,0.6)$),
  with $\xi_{4,2}-\xi_{3,2}=0.3025\ldots$.
  For $(a,b,c)\in \Gamma_1$, $(c-a)/2\leq 0.3$ and $(1-b)/2\leq 0.3$, so
  again $\tilde{M}>M$.

  Finally, for the domain

  $$\Gamma_3~:=~\{ ~a\in [0,b]~,~~~\, b\in [0.48,0.5]~,~~~\, 
  c\in [0.5,0.6]~\} ~,$$
  one obtains $\xi_{4,3}=0.8683113455\ldots$,
  $\xi_{3,3}=0.5662645857\ldots$ and $\xi_{4,3}-\xi_{3,3}=0.3020\ldots$ (for
  $(a,b,c)=(0,0.48,0.5)$ and $(0.5,0.5,0.6)$) whereas
  $(c-a)/2\leq 0.3$ and $(1-b)/2\leq 0.26$. Hence $\tilde{M}>M$ in
  $\Gamma_1\cup \Gamma_2\cup \Gamma_3=T_5$.
  \vspace{1mm}

  Part (1). We deal with the case $b\in [0.25,0.3]$ and we consider first
  the situation when
  $(b,c)\in [0.25,0.3]\times [0.58,0.6]$. With this condition we define the
  domains $\Lambda_j$, $1\leq j\leq 8$,
  by the inequalities $\alpha_{j-1}\leq a\leq \min (\alpha_j,b)$ ($\alpha_0=0$),
  where $\alpha_j$ are equal~to

  $$0.22~,~~~\, 0.25~,~~~\, 0.275~,~~~\, 
  0.288~,~~~\, 0.295~,~~~\, 0.298~,~~~\, 0.299~,~~~\, 0.3.$$
  Set $\Lambda :=\cup_{j=1}^8\Lambda_j$. In $\Lambda_j$, the minimal value
  $\xi_{4,\Lambda ,j}$ of $\xi_4$ is attained
  for $(a,b,c)=(\alpha_{j-1},\max (0.25,\alpha_{j-1}),0.5)$
  and the maximal one $\xi_{3,\Lambda ,j}$
  of $\xi_3$ for
  $(a,b,c)=(\alpha_j,0.3,0.6)$. In all cases the difference
  $\xi_{4,\Lambda ,j}-\xi_{3,\Lambda ,j}$ is
  $>0.375=\max_{\Lambda}((c-a)/2,(1-b)/2)$. Thus $\tilde{M}>M$ in
  $\Lambda$.

  The second situation we consider is when
  $(b,c)\in [0.25,0.3]\times [0.565,0.58]$. The minimal value of $\xi_4$
  is obtained for $(a,b,c)=(0,0.25,0.565)$ and the maximal value of $\xi_3$
  for $(a,b,c)=(0.3,0.3,0.58)$. Their difference is $>0.375$ and again
  one gets $\tilde{M}>M$.

  The third situation is with $(b,c)\in [0.25,0.3]\times [0.5,0.565]$,
  with minimal value of $\xi_4$ for $(a,b,c)=(0,0.25,0.5)$ and maximal
  value of $\xi_3$ for $(a,b,c)=(0.3,0.3,0.565)$. Their difference is
  $>0.375$ and the condition $\tilde{M}>M$ holds true.

  Now we deal with the case $b\in [0.3,0.35]$. It is subdivided between the
  situations $c\in [0.5,0.584]$ and $c\in [0.584,0.6]$. The minimal value of
  $\xi_4$ is attained for $(a,b,c)=[0,0.3,0.5]$ and $[0,0.3,0.584]$.
  The maximal value of $\xi_3$ is obtained for $(a,b,c)=[0.35,0.35,0.584]$
  and $[0.35,0.35,0.6]$. In both cases the lower bound for the difference
  $\xi_4-\xi_3$ is $>0.35$. For
  $(b,c)\in [0.3,0.35]\times [0.5,0.6]$, one has $M\leq 0.35$, i.~e. once again
  the inequality $\tilde{M}>M$ is established.

  \end{proof}

\section{Proving numerically that the case $L-R+$ fails on the set $S_2$
  \protect\label{secnumer}}

In the space $Oabc$ we define the distance by means of the function
$|a|+|b|+|c|$. By $\delta >0$ we denote the step of computation which 
for different parts of $S_2$ will take the values
$10^{-3}$, $5\times 10^{-4}$, $2.5\times 10^{-4}$ and $10^{-4}$.
To fix the ideas we explain first how the computation works for the subset
$S_3\subset S_2$, where

$$S_3~:=~\{ ~(a,b,c)\in \mathbb{R}^3~|~a\in [0,b]~,~b\in [0.25,0.44]~,~
c\in [0.6,1]~\}~,$$
so $S_3$, like $S_2$, is a cylinder over a right trapezoid. The other parts
of $S_3$ will be cylinders either over right trapezoids or over
rectangular cuboids,
see the explanations preceding Fig.~\ref{Yousratab}.
We consider the set $\Theta \subset Oabc$ of points of the form

$$\Theta ~:=~\{ ~A:=(k\times \delta,~\ell \times \delta ,~n\times
\delta )~,~~~\,
k,~\ell ,~n\in \mathbb{Z},~~~\, A\in S_3~\}~.$$
For $S_3$, the step $\delta$ is equal to $10^{-3}$, but the computation is
explained for $\delta$ equal to any of the four possible values.
For each point $A_1:=(\alpha _1,\beta _1,\gamma _1)\in S_3$, there exists a point
$A_2:=(\alpha _2,\beta _2,\gamma _2)\in \Theta$ such that

$$|\alpha_2-\alpha_1|\leq \delta ~,~~~\, |\beta_2-\beta_1|\leq \delta ~,~~~\,
|\gamma_2-\gamma_1|\leq \delta ~.$$
For each point $A\in S_3$, we set

$$\tilde{m}(A):=\min_{i=1,2,3} (\xi_{i+1}(A)-\xi_i(A))~~~\, {\rm and}~~~\,
\tilde{M}(A):=\max_{i=1,2,3} (\xi_{i+1}(A)-\xi_i(A))~.$$
We
show numerically that for $A_2\in \Theta$, one has

\begin{equation}\label{equgrid}
  \tilde{m}(A_2)\geq m(A_2)+6\times \delta ~~~\, {\rm or}~~~\,
\tilde{M}(A_2)\geq M(A_2)+6\times \delta ~.\end{equation}
In the first case $L-$ fails at the point $A_2$,
in the second case one does not have $R+$.

This implies that one does not have $L-R+$ in $S_3$. 
Indeed, one can connect the points $A_1$ and $A_2$ by the piecewise-linear path
$A_1A_3A_4A_2\subset S_3$, where

$$A_3:=(\alpha_2,\beta_1,\gamma_1)~~~\, {\rm and}~~~\,
A_4:=(\alpha_2,\beta_2,\gamma_1)~.$$ 
For each segment $A_{\mu}A_{\nu}$, where $(\mu ,\nu )=(1,3)$,
$(3,4)$ or $(4,2)$, one applies
Proposition~\ref{prop0} to obtain the inequalities

$$|\tilde{m}(A_{\mu})-\tilde{m}(A_{\nu})|\leq \delta ~~~\, {\rm and}~~~\,
|\tilde{M}(A_{\mu})-\tilde{M}(A_{\nu})|\leq \delta ~.$$
Along the path $A_1A_3A_4A_2\subset S_2$ each of the quantities $|z_1|=|b/2|$
and 
$|z_3|=|(1-b)/2|$ varies by not more than $\delta /2$; $|z_2|=|(c-a)/2|$
varies by not more than $\delta$. Thus at $A_1$ one has

$$\begin{array}{lcl}\tilde{m}(A_1)\geq \tilde{m}(A_2)-3\times \delta ~,&&
  \tilde{M}(A_1)\geq \tilde{M}(A_2)-3\times \delta ~,\\ \\
  |m(A_1)-m(A_2)|\leq 2\times \delta&{\rm and}&|M(A_1)-M(A_2)|\leq
  2\times \delta ~,\end{array}$$
so one gets

$$\begin{array}{l}\tilde{m}(A_1)\geq \tilde{m}(A_2)-3\times \delta
  \geq m(A_2)+3\times \delta
  \geq m(A_1)+\delta ~~~\, {\rm or}\\ \\ 
  \tilde{M}(A_1)\geq \tilde{m}(A_2)-3\times \delta \geq M(A_2)+
  3\times \delta \geq M(A_1)+\delta ~.
\end{array}$$
For the other subsets of $S_2$ the choice of the steps $\delta$ is explained
by the following table (see Fig.~\ref{Yousratab} in which the second line
corresponds to the domain~$S_3$).
The inequality $a\leq b$ is everywhere self-understood.
By the letters A, B, C and D we indicate that
$\delta =10^{-3}$, $5\times 10^{-4}$, $2.5\times 10^{-4}$ and $10^{-4}$
respectively. Thus the box in the right lower corner of the table says that
for $b\in [0.48,0.5]$, the computation for the right trapezoid
$(a,c)\in [0.38,~b]\times [0.9,~1]$ is performed with $\delta =10^{-3}$
while for the rectangular cuboids, all with $a\in [0,~0.38]$ and  

$$c\in [0.9,~0.95]~,~~~\, c\in [0.95,~0.99]~~~\, {\rm and}~~~\, c\in [0.99,~1]$$
respectively, the step equals
$5\times 10^{-4}$, $2.5\times 10^{-4}$ and $10^{-4}$.

\begin{figure}[htbp]
\centerline{\hbox{\includegraphics[scale=0.5]{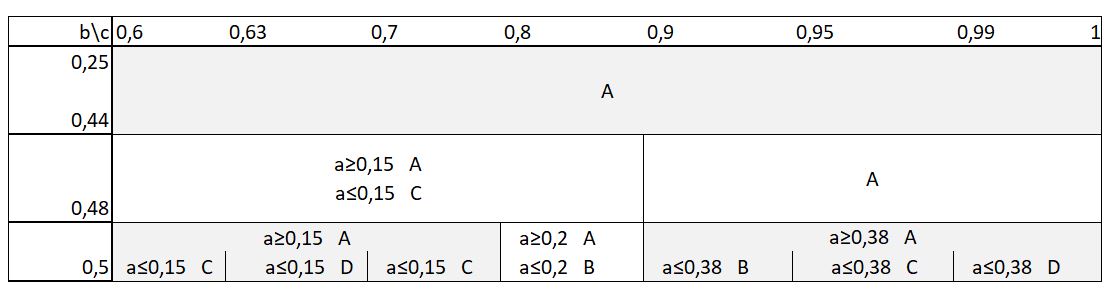}}}
\caption{The choice of the step $\delta$ in the different parts of $S_2$.
  The letters A, B, C and D mean $\delta =10^{-3}$, $5\times 10^{-4}$,
  $2.5\times 10^{-4}$ and $10^{-4}$ respectively.}
\label{Yousratab}
\end{figure}
As explained above, we show numerically
that for every grid point $(a,b,c)$ of $\Theta$,
one has (\ref{equgrid}). We browse the grid points according
to a cycle over $c$ (from the
lowest to the largest value). For each fixed value of $c$, there is a similar
cycle over~$b$, and for each $b$ fixed, there is a cycle over~$a$.
To this end we developed a code in Python (see Algorithm~\ref{algo})
to estimate the values of $\tilde{M}$ and $M$ (respectively of $\tilde{m}$ and $m$)
and calculate the difference between them
for every polynomial having roots $0\leq a\leq b\leq c\leq 1)$.

\begin{algorithm}[H]
  \caption{Algorithm verifying $\tilde{m}(A_2)\ge m(A_2) + 6\delta$
    or $\tilde{M}(A_2) \ge M(A_2) + 6\delta$}
\label{algo}
\begin{algorithmic}[1]
    \STATE Choose a value of $\delta$.
    \STATE Choose the subdomain of $\Theta$ ($c \in [C_1, C_2]$,
    $b\in [B_1,B_2]$, $a\in[A_1,A_2]$ with $C_1\geq 0.6$, $C_2\leq 1$,
    $B_1\geq 0.25$, $B_2\leq 0.5$, $A_1\geq 0$ and $A_2\leq b$).
    \STATE For every grid point $(a,b,c)$ in the subdomain,
    consider the polynomial having the roots $0\leq a\leq b\leq c\leq 1)$,
    calculate the midpoints and find the minimum $m$ and maximum
    $M$ of these midpoints.
    \STATE Differentiate the polynomial and find the roots $\xi_i$.
    \STATE Calculate the minimum $\tilde{m}$ and the maximum $\tilde{M}$
    of $\xi_{i+1}-\xi_i$.
    \STATE Check that $\tilde{m}(A_2)\geq m(A_2) + 6\delta$ or
    $\tilde{M}(A_2) \geq M(A_2) + 6\delta$.
    \STATE Stop and send an error message if the condition is not met.
\end{algorithmic}
\end{algorithm}
The calculations are performed on a computer with an Intel Core
i7-8550U 1.80 GHz processor and 8 GB of RAM. The calculation time is
relatively long. For example, the time needed for the sub-domain
$a\in [0, b]$, $b\in [0.25, 0.44]$ and $c\in[0.6, 1]$ is one hour
and the calculation time is all the more important when the step $\delta$
is low.

\appendix
   \section{Code Source Python} 
This is a code source written in Python for $\delta =0.001$, $c\in[0.6,1]$, $b\in [0.25,0.44]$ and $a\le b$ to check that for each grid point one cannot have $L-R+$. More precisely, the program writes a message error "ERROR: L-R+!" if this happens.	
   \begin{lstlisting}[language=Python, caption={Script Python}]
	
import math
import itertools as it 
from itertools import combinations
import numpy as np
import numpy.polynomial.polynomial as nppol

# here we fix the value of delta
delta=0.001
# calculate the number of grid points
N=math.floor(0.5/delta)+1

for i in range (N+1):
    c= 0.6+i*delta
    if (c <= 1):
        for j  in range (N+1):
            b= 0.25+j*delta
            if (b<= 0.44):
                for k in range (N+1):
                    a=k*delta
                    if (a<=b):
                        #  roots=r1=0, r2=a, r3=b, r4=c ,r5=1
                        r1=0
                        r2=a
                        r3=b
                        r4=c
                        r5=1
                        # calculate midpoints
                        midpointdist=[0,0,0]
                        midpointdist[0]=b/2
                        midpointdist[1]=(c-a)/2
                        midpointdist[2]=(1-b)/2
                        # calculate the minimum and the maximum of midpoints
                        inf = min(midpointdist[0],midpointdist[1],midpointdist[2])
                        sup = max(midpointdist[0],midpointdist[1], midpointdist[2]) 
                        # calculate the polynomial P=x(x-a)(x-b)(x-c)(x-1)
                        P= np.poly((r1, r2, r3, r4, r5))
                        # calculate the derivative of P
                        P_prime = np.polyder(P)
                        # calculate the roots of the derivative of P
                        racine_P_prime=nppol.polyroots(list(reversed(P_prime)))
                        # range them
                        racine_P_prime.sort()
                        # calculate the distances between the roots of the derivative of P
                        neta_moins_ksi=[0,0,0]
                        for i in range (len(racine_P_prime)-1): 
                            neta_moins_ksi[i]=(racine_P_prime[i+1] - racine_P_prime[i])
 #                      # calculate the minimum and the maximum of these distances  
                        inf_neta_moins_ksi=min(neta_moins_ksi[0], neta_moins_ksi[1], neta_moins_ksi[2]) 
                        sup_neta_moins_ksi=max(neta_moins_ksi[0], neta_moins_ksi[1], neta_moins_ksi[2]) 
                        # test if we have L- and R+ 
                        if ((inf >= inf_neta_moins_ksi) and  (sup_neta_moins_ksi <=sup)): print("ERROR: L-R+!")         
print('fin')  
\end{lstlisting}

\end{document}